\documentclass[12pt,twoside]{amsart}
\usepackage{amsmath, amsthm, amscd, amsfonts, amssymb, graphicx}
\usepackage{enumerate}
\usepackage[colorlinks=true,
linkcolor=blue,
urlcolor=cyan,
citecolor=red]{hyperref}
\usepackage{mathrsfs}
\newtheorem{theorem}{Theorem}[section]
\newtheorem{lemma}{Lemma}[section]
\newtheorem{remark}{Remark}[section]

\newtheorem{corollary}{Corollary}[section]

\numberwithin{equation}{section}

\begin{document}
	
\title{Revisiting the Gr\"{u}ss inequality}
\author{H. R. Moradi, S. Furuichi, Z. Heydarbeygi and M. Sababheh}
\subjclass[2010]{Primary 47A63, 26D15, Secondary 47A12, 47A30, 47A64.}
\keywords{ Gr\"{u}ss inequality, arithmetic mean, geometric mean, matrix mean.}

\begin{abstract}
In this article, we explore the celebrated Gr\"{u}ss inequality, where we present a new approach using the Gr\"{u}ss inequality to obtain new refinements of operator means inequalities. We also present several operator Gr\"{u}ss-type inequalities with applications to the numerical radius and entropies.
\end{abstract}
\maketitle
\pagestyle{myheadings}
\markboth{\centerline {Revisiting the Gr\"{u}ss Inequality}}
{\centerline {H. R. Moradi, S. Furuichi, Z. Heydarbeygi and M. Sababheh}}
\bigskip
\bigskip
\section{Introduction}
The celebrated \^Ceby\^sev's inequality \cite{ceby} states that if $h$ and $g$ are two functions having the same monotonicity on $\left[ a,b \right]$, then
\begin{equation}\label{1}
\frac{1}{b-a}\int\limits_{a}^{b}{h\left( t \right)dt}~\frac{1}{b-a}\int\limits_{a}^{b}{g\left( t \right)dt}\le \frac{1}{b-a}\int\limits_{a}^{b}{h\left( t \right)g\left( t \right)dt.}
\end{equation}

Reversing this inequality, Gr\"uss inequality \cite{gruss} states that, for the same $f,g$,
\[\frac{1}{b-a}\int\limits_{a}^{b}{h\left( t \right)g\left( t \right)dt}-\frac{1}{b-a}\int\limits_{a}^{b}{h\left( t \right)dt}\frac{1}{b-a}\int\limits_{a}^{b}{g\left( t \right)dt}\le \frac{1}{4}\left( M-m \right)\left( N-n \right)\]
provided that there exist  real numbers $m$, $M$, $n$, $N$ such that
\[m\le h\left( t \right)\le M\text{ }\And \text{ }n\le g\left( t \right)\le N; \forall a\leq t\leq b.\]

Gr\"{u}ss inequality has received a considerable attention in the literature, as one can see in \cite{ba11,001,002,003,li, mitri2}.

For a complex Hilbert space $\mathscr{H}$, $\mathbb{B}(\mathscr{H})$ will denote the $C^*-$algebra of all bounded operators on $\mathscr{H}$. Upper case letters $A,B$ and $T$ will be used to denote elements in $\mathbb{B}(\mathscr{H})$. When $A\in \mathbb{B}(\mathscr{H}),$ we say that $A$ is positive if $\left<Ax,x\right>>0$, for all non-zero vectors $x\in\mathscr{H}.$

In this article, we are interested in obtaining operator versions of the Gr\"{u}ss inequality and  implementing the Gr\"{u}ss inequality to obtain refinements of some means' inequalities, as a new approach in this direction.


\section{Scalar versions}
The arithmetic-geometric mean inequality (AM-GM inequality) states that
$$\sqrt{ab}\leq \frac{a+b}{2}, a,b>0.$$ The term on the left is called the geometric mean, while the right term is the arithmetic mean of $a$ and $b$. The weighted version of this inequality states that
$$a^{1-v}b^v\leq (1-v)a+vb, \forall 0\leq v\leq 1, a,b>0.$$ This inequality is usually referred to as Young's inequality. For simplicity, we use the notations $$a\sharp_vb:=a^{1-v}b^{v}\;{\text{and}}\;a\nabla_vb=(1-v)a+vb.$$ 
When $v=\frac{1}{2}$, we use $\sharp$ and $\nabla$ instead of $\sharp_{\frac{1}{2}}$ and $\nabla_{\frac{1}{2}}$, respectively. 
Refinements of this inequality have received a considerable attention in the literature, where many forms have been found. We refer the reader to \cite{5, 1, 3, 4, 2} as a sample of such refinements. 

In this article, we present a new approach to refine the AM-GM inequality, resulting in new forms of such refinements. This approach uses the  Gr\"{u}ss inequality.

To better state our results, we remind the reader of the so called  Heron mean, which is defined as follows:
\[{{F}_{t,v}}\left( a,b \right)=\left( 1-t \right)(a\sharp_v b)+t (a\nabla_v b); 0\leq t,v\leq 1.\]

\begin{theorem}\label{thm_1}
Let $a,b\ge 0$. If $g:\left[ 0,1 \right]\to \mathbb{R}$ is  non-decreasing on $\left[ 0,1 \right]$ and $0\leq v\leq 1,$ then
\begin{equation*}
a\sharp_v b+\frac{4}{g(1)-g(0)}\int_{0}^{1}\left(F_{t,v}(a,b)-F_{1/2,v}(a,b)\right)g(t)dt\leq a\nabla_v b.
\end{equation*}
In particular,
\[a\sharp b+\frac{4}{g\left( 1 \right)-g\left( 0 \right)}\left[ \int\limits_{0}^{1}{g\left( t \right){{F}_{t,1/2}}\left( a,b \right)dt}-{{F}_{{1}/{2},1/2\;}}\left( a,b \right)\int\limits_{0}^{1}{g\left( t \right)dt} \right]\le a \nabla b.\]
\end{theorem}
\begin{proof}
If $a,b>0$, then the function $f:\left[ 0,1 \right]\to \mathbb{R}$ defined by 
\[f\left( t \right)=F_{t,v}(a,b)\]
is  non-decreasing on $\left[ 0,1 \right]$. Furthermore,
\[f\left( 0 \right)=a\sharp_v b\text{  }\!\!\And\!\!\text{  }f\left( 1 \right)=a\nabla_v b.\]
Assume that $g$ is a  non-decreasing function on $\left[ 0,1 \right]$.
If we write the inequality \eqref{1} for the functions $f$ and $g$, we get
\[\int\limits_{0}^{1}{F_{t,v}(a,b)dt}\int\limits_{0}^{1}{g\left( t \right)dt}\le \int\limits_{0}^{1}{g\left( t \right)F_{t,v}(a,b)dt},\]
which can be written as
\begin{equation*}
\frac{1}{2}\left( a\sharp_v b+a\nabla_v b \right)\int\limits_{0}^{1}{g\left( t \right)dt}\le \int\limits_{0}^{1}{g\left( t \right)F_{t,v}(a,b)dt}.
\end{equation*}
This means that
$$F_{1/2,v}(a,b)\int_{0}^{1}g(t)dt\leq \int_{0}^{1}g(t)F_{t,v}(a,b)dt.$$
It follows from the Gr\"uss inequality  that
\[\begin{aligned}
   0&\le \int\limits_{0}^{1}{g\left( t \right)F_{t,v}(a,b)dt}-F_{1/2,v}\int\limits_{0}^{1}{g\left( t \right)dt}  
 & \le \frac{1}{4}\left( g\left( 1 \right)-g\left( 0 \right) \right)\left( a\nabla_v b-a\sharp_v b \right).  
\end{aligned}\]
Equivalently,
\[\begin{aligned}
  & a\sharp_v b+\frac{4}{g\left( 1 \right)-g\left( 0 \right)}\int_{0}^{1}\left(F_{t,v}(a,b)-F_{1/2,v}(a,b)\right)g(t)dt\le a\nabla_v b.  
\end{aligned}\]
This proves the first inequality.

Letting $v=\frac{1}{2}$ in the first inequality yields the second inequality and completes the proof.
\end{proof}
\begin{corollary}
Let $a,b\ge 0$. If $g:\left[ 0,1 \right]\to \mathbb{R}$ is  non-decreasing on $\left[ 0,1 \right]$, then
\[\sqrt{ab}\le {{F}_{1/2,{1}/{2}\;}}\left( a,b \right)\le \frac{\int_{0}^{1}{g\left( t \right)}{{F}_{t,1/2}}\left( a,b \right)dt}{\int_{0}^{1}{g\left( t \right)dt}}\le \frac{a+b}{2}.\]
\end{corollary}

Applying Gr\"{u}ss inequality, we obtain the following refinement of the AM-GM inequality, in terms of the Heinz and the logarithmic means. Recall that for two positive numbers $a,b$, the Heinz and logarithmic means are defined, respectively, by
$$H_t(a,b)=\frac{a\sharp_t b+b\sharp_t a}{2},\;0\leq t\leq 1\;{\text{and}}\;L(a,b)=\frac{b-a}{\ln b-\ln a}.$$
\begin{theorem}
Let $g$ be a  non-decreasing function on $\left[ 1/2,1 \right]$. Then for any $a,b> 0$,
\[a\sharp b+\frac{2}{g\left( 1 \right)-g\left( \frac{1}{2} \right)}\left[ \int\limits_{\frac{1}{2}}^{1}{g\left( t \right)H_t(a,b)dt}-L(a,b)\cdot\int\limits_{\frac{1}{2}}^{1}{g\left( t \right)dt} \right]\le a\nabla b.\]
\end{theorem}
\begin{proof}
For $x>0$, define
\[f\left( t \right)=\frac{{{x}^{t}}+{{x}^{1-t}}}{2},\text{ } t\in \left[ \frac{1}{2},1 \right].\]
This function is  non-decreasing on $\left[ 1/2,1 \right]$. Furthermore,
\[f\left( \frac{1}{2} \right)=\sqrt{x}\text{ }\And \text{ }f\left( 1 \right)=\frac{1+x}{2}.\]
Assume that $g$ is a  non-decreasing function on $\left[ 1/2,1 \right]$. If we write the inequality \eqref{1} for the functions $f$ and $g$, we get
\[\int\limits_{\frac{1}{2}}^{1}{\frac{{{x}^{t}}+{{x}^{1-t}}}{2}dt}\cdot\int\limits_{\frac{1}{2}}^{1}{g\left( t \right)dt}\le \frac{1}{2}\int\limits_{\frac{1}{2}}^{1}{g\left( t \right)\frac{{{x}^{t}}+{{x}^{1-t}}}{2}dt}.\]
or equivalently
\[\left( \frac{x-1}{2\ln x} \right)\cdot\int\limits_{\frac{1}{2}}^{1}{g\left( t \right)dt}\le \frac{1}{2}\int\limits_{\frac{1}{2}}^{1}{g\left( t \right)\frac{{{x}^{t}}+{{x}^{1-t}}}{2}dt}.\]
It follows from the Gr\"uss inequality  that
\[\frac{1}{2}\int\limits_{\frac{1}{2}}^{1}{g\left( t \right)\frac{{{x}^{t}}+{{x}^{1-t}}}{2}dt}-\left( \frac{x-1}{2\ln x} \right)\cdot\int\limits_{\frac{1}{2}}^{1}{g\left( t \right)dt}\le \left( \frac{g\left( 1 \right)-g\left( \frac{1}{2} \right)}{4} \right)\left( \frac{1+x}{2}-\sqrt{x} \right).\]
Therefore,
\[\sqrt{x}+\frac{2}{g\left( 1 \right)-g\left( \frac{1}{2} \right)}\left[ \int\limits_{\frac{1}{2}}^{1}{g\left( t \right)\frac{{{x}^{t}}+{{x}^{1-t}}}{2}dt}-\left( \frac{x-1}{\ln x} \right)\cdot\int\limits_{\frac{1}{2}}^{1}{g\left( t \right)dt} \right]\le \frac{1+x}{2}.\]
Replacing $x$ by $\frac{b}{a}$, we obtain the desired inequality. 
\end{proof}

If we take $g\left( t \right)=t$ in Theorem \ref{2}, we get
\begin{corollary}\label{3}
For any $x\ge 0$,
\[\sqrt{x}+\frac{4}{{{\ln }^{2}}x}\left( \frac{1}{8}\left( x-1 \right)\ln x+\sqrt{x}-\frac{x+1}{2} \right)\le \frac{1+x}{2}.\]
\end{corollary}

Corollary \ref{3} implies the following refined arithmetic-geometric mean inequality with the logarithmic mean. 
\begin{corollary}
For any $a,b>0$,
$$
a\sharp b+\gamma(a,b) \cdot L(a,b) \le a\nabla b,
$$
where
$$
\gamma(a,b):=\frac{\ln^2 b/a}{2(\ln^2 b/a +4)} \geq 0.
$$
\end{corollary}
\begin{proof}
From Corollary \ref{3}, we have
\begin{equation}\label{6}
\frac{\ln^2 x}{2(\ln^2 x+4)} \frac{x-1}{\ln x} \leq \frac{x+1}{2}-\sqrt{x}.
\end{equation}
Replacing $x$ by $\frac{b}{a}$ implies the desired inequality and  completes the proof.
\end{proof}

It is interesting to compare \eqref{6} with the following inequality \cite{zou}:
\begin{equation}\label{7}
\frac{\ln ^2 x}{8}\sqrt{x} \leq \frac{x+1}{2} -\sqrt{x}.
\end{equation}
However, there is no ordering between L.H.S. in \eqref{6} and L.H.S. in \eqref{7},
since we have
$$
\sqrt{x} \le \frac{x-1}{\ln x},\quad \frac{1}{2(\ln^2x+4)}\le \frac{1}{8}
$$
for $x>0$. Actually, for a small $x>0$, we have the ordering
$$
\frac{1}{2(\ln^2x+4)}\frac{x-1}{\ln x} \le \frac{1}{8}\sqrt{x},
$$
but we have the opposite inequality for a large $x>0$, for example $x > 11288$.

\section{Non-commutative versions that follow from the scalar ones}
In this section, we present some non-commutative versions for the scalar inequalities we have shown earlier.   The arithmetic and geometric means of two positive $A,B\in\mathbb{B}(\mathscr{H})$ are defined, respectively, by
$$A\nabla_v B=(1-v)A+VB\;{\text{and}}\;A\sharp_v B=A^{\frac{1}{2}}\left(A^{-\frac{1}{2}}BA^{-\frac{1}{2}}\right)^vA^{\frac{1}{2}}, 0\leq v\leq 1.$$ Similar to the scalar case, we have the so called operator arithmetic geometric mean inequality
$$A\sharp_v B\leq A\nabla_v B, A,B\in\mathbb{B}(\mathscr{H})\;{\text{being}}\;positive\;{\text{and}}\;0\leq v\leq 1.$$
Refining the operator AM-GM inequality has received a considerable interest in the literature, as one can see in \cite{5,3,4,zou}. In the next result, we present a new type of such refinements, where we employ Gr\"{u}ss inequality. 
The first result, is the following operator version of Theorem \ref{thm_1}, in which we still adopt the notation 
$$F_{t,v}(A,B)=(1-t)(A\sharp_v B)+t (A\nabla_v B);$$ as the operator weighted Heron mean of the positive operators $A,B$.
\begin{theorem}\label{thm_op_1}
Let $A,B\in\mathbb{B}(\mathscr{H})$ be positive operators and let $0\leq v\leq 1.$  If $g:\left[ 0,1 \right]\to \mathbb{R}$ is  non-decreasing on $\left[ 0,1 \right]$, then
\begin{equation*}
A\sharp_v B+\frac{4}{g(1)-g(0)}\int_{0}^{1}\left(F_{t,v}(A,B)-F_{1/2,v}(A,B)\right)g(t)dt\leq A\nabla_v B.
\end{equation*}
\end{theorem}
\begin{proof}
Letting $a=1$ in Theorem \ref{thm_1}, we have
$$b^v+\frac{4}{g(1)-g(0)}\int_{0}^{1}\left(\left\{(1-t)b^v+t(1-v+vb)\right\}-\frac{1}{2}\left(\sqrt{b}+\frac{1+b}{2}\right)\right)g(t)dt\leq 1-v+vb.$$ Applying a standard functional calculus argument with $b=A^{-\frac{1}{2}}BA^{-\frac{1}{2}}$, then multiplying both sides of the inequality by $A^{\frac{1}{2}}$ imply the desired inequality.
\end{proof}

On the other hand, an operator version of Theorem \ref{2} may be stated as follows. The proof is similar to that of Theorem \ref{thm_op_1}, hence is not included.
\begin{theorem}\label{2}
Let $A,B\in\mathbb{B}(\mathscr{H})$ be positive and let $g$ be a  non-decreasing function on $\left[ 1/2,1 \right]$. Then 
\[A\sharp B+\frac{2}{g\left( 1 \right)-g\left( \frac{1}{2} \right)}\left[ \int\limits_{\frac{1}{2}}^{1}g\left( t \right)\frac{{A\sharp_t B}+{A\sharp_{1-t}B}}{2}dt-(B-A)S_0(A|B)^{-1}A\cdot\int\limits_{\frac{1}{2}}^{1}{g\left( t \right)dt} \right]\le A\nabla B,\]
where $S_0(A|B)=A^{1/2}\log \left(A^{-1/2}BA^{-1/2}\right) A^{1/2}$ is the relative operator entropy of the positive operators $A,B$ \cite{fujii}.
\end{theorem}

For the next  result, we  define
	\[A{{m}_{t,v}}B={{A}^{\frac{1}{2}}}{{\left( \left( 1-v \right)I+v{{\left( {{A}^{-\frac{1}{2}}}B{{A}^{-\frac{1}{2}}} \right)}^{t}} \right)}^{\frac{1}{t}}}{{A}^{\frac{1}{2}}},\]
for the positive $A,B\in\mathbb{B}(\mathscr{H})$, $-1\le t\le 1$ and $0\le v\le 1$. In this result, we present a refinement of the operator AM-GM inequality, without using a functional calculus argument.
\begin{theorem}
Let $A,B\in \mathbb{B}\left( \mathscr{H} \right)$ be two positive operators. If $g:\left[ 0,1 \right]\to \mathbb{R}$ is  non-decreasing on $\left[ 0,1 \right]$, then
\[A{{\sharp}_{v}}B+\frac{4}{g\left( 1 \right)-g\left( 0 \right)}\left[ \int\limits_{0}^{1}{\left( A{{m}_{t,v}}B \right)g\left( t \right)dt}-\int\limits_{0}^{1}{\left( A{{m}_{t,v}}B \right)dt}\int\limits_{0}^{1}{g\left( t \right)dt} \right]\le A{{\nabla }_{v}}B.\]
\end{theorem}
\begin{proof}
Define 
	\[f\left( t \right)=\left\langle \left( A{{m}_{t,v}}B \right)x,x \right\rangle ,\text{ for any }x\in \mathscr{H}.\]
Of course, $f$ is  non-decreasing on $\left[ -1,1 \right]$ (since $A{{m}_{t,v}}B$ is an operator mean).
In particular, we have
	\[f\left( -1 \right)=\left\langle \left( A{{!}_{v}}B \right)x,x \right\rangle \text{, }f\left( 0 \right)=\left\langle \left( A{{\sharp}_{v}}B \right)x,x \right\rangle ,\text{ }f\left( 1 \right)=\left\langle \left( A{{\nabla }_{v}}B \right)x,x \right\rangle ,\] where $A!_vB=((1-v)A^{-1}+vB^{-1})^{-1}$ is the harmonic mean of $A,B$.
From the inequality \eqref{1}, we have
\[\int\limits_{0}^{1}{\left\langle \left( A{{m}_{t,v}}B \right)x,x \right\rangle dt}\int\limits_{0}^{1}{g\left( t \right)dt}\le \int\limits_{0}^{1}{\left\langle \left( A{{m}_{t,v}}B \right)x,x \right\rangle g\left( t \right)dt},\]
which is equivalent to
\[\left\langle \left( \int\limits_{0}^{1}{\left( A{{m}_{t,v}}B \right)dt}\int\limits_{0}^{1}{g\left( t \right)dt} \right)x,x \right\rangle \le \left\langle \left( \int\limits_{0}^{1}{\left( A{{m}_{t,v}}B \right)g\left( t \right)dt} \right)x,x \right\rangle.\]
Now, Gr\"uss inequality implies
\[\begin{aligned}
  & \left\langle \left[ \int\limits_{0}^{1}{\left( A{{m}_{t,v}}B \right)g\left( t \right)dt}-\int\limits_{0}^{1}{\left( A{{m}_{t,v}}B \right)dt}\int\limits_{0}^{1}{g\left( t \right)dt} \right]x,x \right\rangle  \\ 
 & \le \left\langle \left[ \left( \frac{g\left( 1 \right)-g\left( 0 \right)}{4} \right)\left( A{{\nabla }_{v}}B-A{{\sharp}_{v}}B \right) \right]x,x \right\rangle,
\end{aligned}\]
for any vector $x \in \mathscr{H}$. Therefore we obtain
\[A{{\sharp}_{v}}B+\frac{4}{g\left( 1 \right)-g\left( 0 \right)}\left[ \int\limits_{0}^{1}{\left( A{{m}_{t,v}}B \right)g\left( t \right)dt}-\int\limits_{0}^{1}{\left( A{{m}_{t,v}}B \right)dt}\int\limits_{0}^{1}{g\left( t \right)dt} \right]\le A{{\nabla }_{v}}B.\]
Therefore the desire inequality is obtained.
\end{proof}

On the other hand, a refinement of the operator geometric-harmonic mean inequality can be stated as follows. The proof is similar to the above arguments, and hence we omit it.
\begin{theorem}
Let $A,B\in \mathbb{B}\left( \mathscr{H} \right)$ be two positive operators. If $g:\left[ -1,0 \right]\to \mathbb{R}$ is  non-decreasing on $\left[ -1,0 \right]$, then
\[A{{!}_{v}}B+\frac{4}{g\left( 0 \right)-g\left( -1 \right)}\left[ \int\limits_{-1}^{0}{\left( A{{m}_{t,v}}B \right)g\left( t \right)dt}-\int\limits_{-1}^{0}{\left( A{{m}_{t,v}}B \right)dt}\int\limits_{-1}^{0}{g\left( t \right)dt} \right]\le A{{\sharp}_{v}}B\]
\end{theorem}

We conclude this section by presenting the following application towards relative operator entropies.
\begin{theorem}\label{theorem4.5}
Let $A,B\in\mathbb{B}(\mathscr{H})$ be positive and $0<s<1$. Then
\[ S_0(A|B)+2\int\limits_0^1(2t-1)S_{st}(A|B)dt \le S_s(A|B),\]
where $S_p(A|B):=A^{1/2}\ln_p \left(A^{-1/2}BA^{-1/2}\right) A^{1/2}$
is Tsallis relative operator entropy \cite{FYK2005} and $S_0(A|B)=\lim\limits_{p\to 0}S_p(A|B)=A^{1/2}\log \left(A^{-1/2}BA^{-1/2}\right) A^{1/2}$ is relative operator entropy.
\end{theorem}
\begin{proof}
Define 
	\[f\left( t \right)=\frac{{{x}^{ts}}-1}{ts},\quad\text{ }x>0,~0\le s\le 1, t\in [0,1].\]
 Then
	\[f\left( 0 \right)=\log x\text{ and }f\left( 1 \right)=\frac{{{x}^{s}}-1}{s}.\]
Now, from the Gr\"{u}ss inequality
\begin{equation*}\label{theorem4.5_eq01}
\int\limits_{0}^{1}{\frac{{{x}^{ts}}-1}{ts}g\left( t \right)dt}-\int\limits_{0}^{1}{\frac{{{x}^{ts}}-1}{ts}dt}\int\limits_{0}^{1}{g\left( t \right)dt}\le \left( \frac{g\left( 1 \right)-g\left( 0 \right)}{4} \right)\left( \frac{{{x}^{s}}-1}{s}-\log x \right).
\end{equation*}
Namely,
	\[\log x+\frac{4}{g\left( 1 \right)-g\left( 0 \right)}\left[ \int\limits_{0}^{1}{\frac{{{x}^{ts}}-1}{ts}g\left( t \right)dt}-\int\limits_{0}^{1}{\frac{{{x}^{ts}}-1}{ts}dt}\int\limits_{0}^{1}{g\left( t \right)dt} \right]\le \frac{{{x}^{s}}-1}{s}.\]
If we set $g(t):=2t$, then the above inequality is written by
\[\log x +2\int\limits_0^1\frac{(2t-1)(x^{st}-1)}{st}dt \le \ln_s x,\]
where $\ln_s x:=\frac{x^s-1}{s}$. Applying functional calculus argument in the above inequality implies 
\[S_0(A|B)+2\int\limits_0^1(2t-1)S_{st}(A|B)dt \le S_s(A|B),\]
where $S_p(A|B):=A^{1/2}\ln_p \left(A^{-1/2}BA^{-1/2}\right) A^{1/2}$
is Tsallis relative operator entropy and $S_0(A|B)=\lim\limits_{p\to 0}S_p(A|B)=A^{1/2}\log \left(A^{-1/2}BA^{-1/2}\right) A^{1/2}$ is relative operator entropy. This completes the proof.
\end{proof}

Theorem \ref{theorem4.5} gives a refinement of $S_0(A|B) \leq S_s(A|B)$ shown in \cite[Proposition 3.1]{FYK2005}.

\section{Sharpening Gr\"{u}ss inequality and covariance versions}

We conclude this article by presenting some covariance inequalities that are of  Gr\"{u}ss type, with an application to the numerical radius.
\begin{theorem}
Let $T\in \mathbb{B}\left( \mathscr{H} \right)$ and $x\in \mathscr{H}$ be a unit vector. Then
\[\left| \left\langle \left| T \right|\left| {{T}^{*}} \right|x,x \right\rangle -\left\langle \left| T \right|x,x \right\rangle \left\langle \left| {{T}^{*}} \right|x,x \right\rangle  \right|\le \left\| \left| T \right|x \right\|\left\| \left| {{T}^{*}} \right|x \right\|-\left| \left\langle Tx,x \right\rangle  \right|^2.\]
\end{theorem}
\begin{proof}
Let $A,B\in \mathbb{B}\left( \mathscr{H} \right)$ be two positive operators and $x\in \mathscr{H}$ be a unit vector. Then
\begin{align}
   \left| \left\langle ABx,x \right\rangle -\left\langle Ax,x \right\rangle \left\langle Bx,x \right\rangle  \right|&=\left| \left\langle \left( B-\left\langle Bx,x \right\rangle I \right)x,\left( A-\left\langle Ax,x \right\rangle I \right)x \right\rangle  \right| \nonumber\\ 
 & \le \left\| \left( A-\left\langle Ax,x \right\rangle I \right)x \right\|\left\| \left( B-\left\langle Bx,x \right\rangle I \right)x \right\| \nonumber\\ 
 & ={{\left( \left\langle {{A}^{2}}x,x \right\rangle -{{\left\langle Ax,x \right\rangle }^{2}} \right)}^{\frac{1}{2}}}{{\left( \left\langle {{B}^{2}}x,x \right\rangle -{{\left\langle Bx,x \right\rangle }^{2}} \right)}^{\frac{1}{2}}} \label{x1}\\ 
 & \le \sqrt{\left\langle {{A}^{2}}x,x \right\rangle \left\langle {{B}^{2}}x,x \right\rangle }-\left\langle Ax,x \right\rangle \left\langle Bx,x \right\rangle \label{x2}, 
\end{align}
where \eqref{x2} follows from the inequality $\left( {{a}^{2}}-{{b}^{2}} \right)\left( {{c}^{2}}-{{d}^{2}} \right)\le {{\left( ac-bd \right)}^{2}}$, $a,b,c,d\in \mathbb{R}$. Notice that \eqref{x1} is meaningful, since for any self-adjoint operator $X\in \mathbb{B}\left( \mathscr{H} \right)$, we have
\[{{\left\langle Xx,x \right\rangle }^{2}}\le \left\langle {{X}^{2}}x,x \right\rangle.\]
Therefore,
\begin{equation}\label{x3}
\left\langle Ax,x \right\rangle \left\langle Bx,x \right\rangle +\left| \left\langle ABx,x \right\rangle -\left\langle Ax,x \right\rangle \left\langle Bx,x \right\rangle  \right|\le \sqrt{\left\langle {{A}^{2}}x,x \right\rangle \left\langle {{B}^{2}}x,x \right\rangle }.
\end{equation}
Now, replacing $A$ and $B$ by $\left| T \right|$ and $\left| {{T}^{*}} \right|$, respectively, then we get
\[\left\langle \left| T \right|x,x \right\rangle \left\langle \left| {{T}^{*}} \right|x,x \right\rangle +\left| \left\langle \left| T \right|\left| {{T}^{*}} \right|x,x \right\rangle -\left\langle \left| T \right|x,x \right\rangle \left\langle \left| {{T}^{*}} \right|x,x \right\rangle  \right|\le \sqrt{\left\langle {{\left| T \right|}^{2}}x,x \right\rangle \left\langle {{\left| {{T}^{*}} \right|}^{2}}x,x \right\rangle }.\]
On the other hand, since (see e.g., \cite[pp. 75--76]{halmos})
\[\left| \left\langle Tx,x \right\rangle  \right|\le \sqrt{\left\langle \left| T \right|x,x \right\rangle \left\langle \left| {{T}^{*}} \right|x,x \right\rangle },\]
we infer that
\[{{\left| \left\langle Tx,x \right\rangle  \right|}^{2}}+\left| \left\langle \left| T \right|\left| {{T}^{*}} \right|x,x \right\rangle -\left\langle \left| T \right|x,x \right\rangle \left\langle \left| {{T}^{*}} \right|x,x \right\rangle  \right|\le \sqrt{\left\langle {{\left| T \right|}^{2}}x,x \right\rangle \left\langle {{\left| {{T}^{*}} \right|}^{2}}x,x \right\rangle },\]
as desired
\end{proof}
As an application, we present the following numerical radius inequality that refines the celebrated Kittaneh result in \cite{kittanehnumerical}.
\begin{corollary}
Let $T\in \mathbb{B}\left( \mathscr{H} \right)$. Then
\[{{\omega }^{2}}\left( T \right)+\underset{\left\| x \right\|=1}{\mathop{\underset{x\in \mathscr{H}}{\mathop{\inf }}\,}}\,\left\{ \left| \left\langle \left| T \right|\left| {{T}^{*}} \right|x,x \right\rangle -\left\langle \left| T \right|x,x \right\rangle \left\langle \left| {{T}^{*}} \right|x,x \right\rangle  \right| \right\}\le \frac{1}{2}\left\| {{\left| T \right|}^{2}}+{{\left| {{T}^{*}} \right|}^{2}} \right\|.\]
\end{corollary}
\begin{proof}
Applying the arithmetic-geometric mean inequality, we have
\[{{\left| \left\langle Tx,x \right\rangle  \right|}^{2}}+\left| \left\langle \left| T \right|\left| {{T}^{*}} \right|x,x \right\rangle -\left\langle \left| T \right|x,x \right\rangle \left\langle \left| {{T}^{*}} \right|x,x \right\rangle  \right|\le \left\langle \left( \frac{{{\left| T \right|}^{2}}+{{\left| {{T}^{*}} \right|}^{2}}}{2} \right)x,x \right\rangle.\]
Consequently, by taking supremum over all unit vector $x\in \mathscr{H}$, we get
\[{{\omega }^{2}}\left( A \right)+\underset{\left\| x \right\|=1}{\mathop{\underset{x\in \mathscr{H}}{\mathop{\inf }}\,}}\,\left\{ \left| \left\langle \left| T \right|\left| {{T}^{*}} \right|x,x \right\rangle -\left\langle \left| T \right|x,x \right\rangle \left\langle \left| {{T}^{*}} \right|x,x \right\rangle  \right| \right\}\le \frac{1}{2}\left\| {{\left| T \right|}^{2}}+{{\left| {{T}^{*}} \right|}^{2}} \right\|.\]
This completes the proof.
\end{proof}

\begin{remark}\label{rem11}
From the inequality \eqref{x3}, we obtain the covariance inequality
\[\left\langle Ax,x \right\rangle \left\langle Bx,x \right\rangle -\left| \left\langle ABx,x \right\rangle  \right|\le \sqrt{\left\langle {{A}^{2}}x,x \right\rangle \left\langle {{B}^{2}}x,x \right\rangle }-\left\langle Ax,x \right\rangle \left\langle Bx,x \right\rangle,\]
for the positive operators $A,B\in\mathbb{B}(\mathscr{H})$.
Thus,
\[\left\langle Ax,x \right\rangle \left\langle Bx,x \right\rangle \le \frac{\sqrt{\left\langle {{A}^{2}}x,x \right\rangle \left\langle {{B}^{2}}x,x \right\rangle }+\left| \left\langle ABx,x \right\rangle  \right|}{2},\]
which implies
\[\begin{aligned}
  {{\left\langle Ax,x \right\rangle }^{2}}{{\left\langle Bx,x \right\rangle }^{2}}& \le {{\left( \frac{\sqrt{\left\langle {{A}^{2}}x,x \right\rangle \left\langle {{B}^{2}}x,x \right\rangle }+\left| \left\langle ABx,x \right\rangle  \right|}{2} \right)}^{2}} \\ 
 & \le \frac{\left\langle {{A}^{2}}x,x \right\rangle \left\langle {{B}^{2}}x,x \right\rangle +{{\left| \left\langle ABx,x \right\rangle  \right|}^{2}}}{2} \\ 
 & \le \left\langle {{A}^{2}}x,x \right\rangle \left\langle {{B}^{2}}x,x \right\rangle.   
\end{aligned}\]
This provides two refining terms of the celebrated inequality 
$$ {{\left\langle Ax,x \right\rangle }^{2}}{{\left\langle Bx,x \right\rangle }^{2}}\leq \left\langle {{A}^{2}}x,x \right\rangle \left\langle {{B}^{2}}x,x \right\rangle. $$
\end{remark}
We conclude this article by presenting some covariance inequalities similar to Remark \ref{rem11}, but in a more elaborated form. First, a scalar inequality.
\begin{lemma}\label{y1}
Let $a,b,c,d>0$. Then
\[\frac{1}{2}\frac{{{\left( {{a}^{2}}{{d}^{2}}-{{b}^{2}}{{c}^{2}} \right)}^{2}}}{{{a}^{2}}{{d}^{2}}+{{b}^{2}}{{c}^{2}}}+\left( {{a}^{2}}-{{b}^{2}} \right)\left( {{c}^{2}}-{{d}^{2}} \right)\le {{\left( ac-bd \right)}^{2}}.\]
\end{lemma}
\begin{proof}
Since
\[{{\left( \frac{a+b}{2} \right)}^{2}}-ab=\left( \frac{a+b}{2}-\sqrt{ab} \right)\left( \frac{a+b}{2}+\sqrt{ab} \right),\]
we have
\[\frac{{{\left( \frac{a+b}{2} \right)}^{2}}-ab}{\frac{a+b}{2}+\sqrt{ab}}=\frac{a+b}{2}-\sqrt{ab}.\]
Equivalently,
\[{{\left( \frac{\frac{a+b}{2}+\sqrt{ab}}{2} \right)}^{-1}}\frac{{{\left( \frac{a+b}{2} \right)}^{2}}-ab}{2}=\frac{a+b}{2}-\sqrt{ab}.\]
Now, by applying the arithmetic-geometric mean inequality, we obtain
\[\frac{1}{4}\frac{{{\left( a-b \right)}^{2}}}{a+b}\le \frac{a+b}{2}-\sqrt{ab}\le \frac{1}{8}\frac{{{\left( a-b \right)}^{2}}}{\sqrt{ab}}.\]
Rearranging the terms, we get
\[\frac{1}{2}\frac{{{\left( {{a}^{2}}{{d}^{2}}-{{b}^{2}}{{c}^{2}} \right)}^{2}}}{{{a}^{2}}{{d}^{2}}+{{b}^{2}}{{c}^{2}}}+\left( {{a}^{2}}-{{b}^{2}} \right)\left( {{c}^{2}}-{{d}^{2}} \right)\le {{\left( ac-bd \right)}^{2}},\]
as desired.
\end{proof}
\begin{theorem}\label{13}
Let $A,B\in \mathbb{B}\left( \mathscr{H} \right)$ be positive operators such that $mI\le A \le MI$ $nI\le B \le NI$, for some positive scalars $m,M,n,N$. Then for any unit vector $x \in \mathscr{H}$,
\[\begin{aligned}
  & \left| \left\langle ABx,x \right\rangle -\left\langle Ax,x \right\rangle \left\langle Bx,x \right\rangle  \right| \\ 
 & \le \frac{\left( M-m \right)\left( N-n \right)}{4}-\left( \sqrt{\mathscr{C}\left( A,x \right)\mathscr{C}\left( B,x \right)}+\frac{{{\left( {{\left( M-m \right)}^{2}}\mathscr{C}\left( B,x \right)-{{\left( N-n \right)}^{2}}\mathscr{C}\left( A,x \right) \right)}^{2}}}{8{{\left( M-m \right)}^{2}}\mathscr{C}\left( B,x \right)+{{\left( N-n \right)}^{2}}\mathscr{C}\left( A,x \right)} \right), 
\end{aligned}\]
where
\[\mathscr{C}\left( A,x \right)=\left\langle \left( M-A \right)\left( A-m \right)x,x \right\rangle \text{ and }\mathscr{C}\left( B,x \right)=\left\langle \left( N-B \right)\left( B-n \right)x,x \right\rangle.\]
\end{theorem}
\begin{proof}
It has been shown in \eqref{x1} that
\[\left| \left\langle ABx,x \right\rangle -\left\langle Ax,x \right\rangle \left\langle Bx,x \right\rangle  \right|\le \left( \left\langle {{A}^{2}}x,x \right\rangle -{{\left\langle Ax,x \right\rangle }^{2}} \right)\left( \left\langle {{B}^{2}}x,x \right\rangle -{{\left\langle Bx,x \right\rangle }^{2}} \right).\]
By the arithmetic-geometric mean inequality, we have
\begin{equation}\label{y2}
\begin{aligned}
  & \left\langle A^2x,x\right\rangle-{{\left\langle Ax,x \right\rangle }^{2}} \\ 
 & =\left( M-\left\langle Ax,x \right\rangle  \right)\left( \left\langle Ax,x \right\rangle -m \right)-\left\langle \left( MI-A \right)\left( A-mI \right)x,x \right\rangle  \\ 
 & \le {{\left( \frac{M-m}{2} \right)}^{2}}-\left\langle \left( MI-A \right)\left( A-mI \right)x,x \right\rangle,
\end{aligned}
\end{equation}
and similarly
\[\begin{aligned}
  \left\langle B^2x,x\right\rangle-{{\left\langle Bx,x \right\rangle }^{2}} 
  \le {{\left( \frac{N-n}{2} \right)}^{2}}-\left\langle \left( NI-B \right)\left( B-nI \right)x,x \right\rangle.
\end{aligned}\]
Now, by applying Lemma \ref{y1}, we get
\[\begin{aligned}
  & \left| \left\langle ABx,x \right\rangle -\left\langle Ax,x \right\rangle \left\langle Bx,x \right\rangle  \right| \\ 
 & \le \sqrt{\left( {{\left( \frac{M-m}{2} \right)}^{2}}-\mathscr{C}\left( A,x \right) \right)\left( {{\left( \frac{N-n}{2} \right)}^{2}}-\mathscr{C}\left( B,x \right) \right)} \\ 
 & \le \frac{\left( M-m \right)\left( N-n \right)}{4}-\left( \sqrt{\mathscr{C}\left( A,x \right)\mathscr{C}\left( B,x \right)}+\frac{{{\left( {{\left( M-m \right)}^{2}}\mathscr{C}\left( B,x \right)-{{\left( N-n \right)}^{2}}\mathscr{C}\left( A,x \right) \right)}^{2}}}{8\left( {{\left( M-m \right)}^{2}}\mathscr{C}\left( B,x \right)-{{\left( N-n \right)}^{2}}\mathscr{C}\left( A,x \right) \right)} \right). 
\end{aligned}\]
This completes the proof of the theorem.
\end{proof}
\begin{remark}
Since
\[\left( NI-B \right)\left( B-nI \right)={{\left( \frac{N-n}{2} \right)}^{2}}I-{{\left| B-\frac{N+n}{2}I \right|}^{2}},\]
and
\[\left( MI-A \right)\left( A-mI \right)={{\left( \frac{M-m}{2} \right)}^{2}}I-{{\left| A-\frac{M+m}{2}I \right|}^{2}},\]
we infer from \eqref{y2} that
\[\left\langle {{A}^{2}}x,x \right\rangle -{{\left\langle Ax,x \right\rangle }^{2}}\le \left\langle {{\left| A-\frac{M+m}{2}I \right|}^{2}}x,x \right\rangle \]
and
\[\left\langle {{B}^{2}}x,x \right\rangle -{{\left\langle Bx,x \right\rangle }^{2}}\le \left\langle {{\left| B-\frac{N+n}{2}I \right|}^{2}}x,x \right\rangle. \]
This in turns implies that
\[\left| \left\langle ABx,x \right\rangle -\left\langle Ax,x \right\rangle \left\langle Bx,x \right\rangle  \right|\le \left\| A-\frac{M+m}{2}I \right\|\left\| B-\frac{N+n}{2}I \right\|.\]
Since $mI\le A\le MI$ and $nI\le B\le NI$, then 
	\[\left| \left\langle \left( A-\frac{M+m}{2}I \right)x,x \right\rangle  \right|\le \frac{M-m}{2},\]
and
	\[\left| \left\langle \left( B-\frac{N+n}{2}I \right)x,x \right\rangle  \right|\le \frac{N-n}{2}.\]
The above relations imply
\begin{equation*}\label{10}
\left\| A-\frac{M+m}{2}I \right\|\le \frac{M-m}{2},
\end{equation*}
and
\begin{equation*}\label{11}
\left\| B-\frac{N+n}{2}I \right\|\le \frac{N-n}{2}.
\end{equation*}
Consequently,
\[\begin{aligned}
   \left| \left\langle ABx,x \right\rangle -\left\langle Ax,x \right\rangle \left\langle Bx,x \right\rangle  \right|\le \left\| A-\frac{M+m}{2}I \right\|\left\| B-\frac{N+n}{2}I \right\|  \le \frac{\left( M-m \right)\left( N-n \right)}{4}.  
\end{aligned}\]
\end{remark}


\section*{Acknowledgement}
The author (S.F.) was partially supported by JSPS KAKENHI Grant Number 16K05257.

{\tiny (H. R. Moradi) Department of Mathematics, Payame Noor University (PNU), P.O. Box 19395-4697, Tehran, Iran}

{\tiny \textit{E-mail address:} hrmoradi@mshdiau.ac.ir }

\medskip

{\tiny (S. Furuichi) 
Department of Information Science, College of Humanities and Sciences, Nihon University, 3-25-40, Sakurajyousui, Setagaya-ku,
Tokyo, 156-8550, Japan}

{\tiny \textit{E-mail address:} furuichi@chs.nihon-u.ac.jp}

\medskip
{\tiny (Z. Heydarbeygi) Department of Mathematics, Payame Noor University (PNU), P.O. Box 19395-4697, Tehran, Iran}

{\tiny \textit{E-mail address:} zheydarbeygi@yahoo.com}

\medskip
{\tiny (M. Sababheh) Department of basic sciences, Princess Sumaya University for Technology, Amman 11941, Jordan}

{\tiny \textit{E-mail address:} sababheh@psut.edu.jo}


\begin{thebibliography}{9}

\bibitem{ba11}
S. Ballasubramanian, {\it On the Gr\"{u}ss inequality for unital 2-positive linear maps}, Oper. Matrices. {\bf{ 10}}(3) (2016), 643--649.

\bibitem{ceby}
P. L. \^Ceby\^sev, {\it Sur les expressions approximatatives des int\'egrales d\'efinies par les autres prises entre les m\'eme limites}, Proc. Math. Soc. Kharkov, {\bf2} (1882), 93--98 (Russian), translated in Oeuvres, {\bf2} (1907), 716--719.


\bibitem{001}
S. S. Dragomir, {\it Some Gr\"uss type inequalities in inner product spaces}, J. Inequal. Pure Appl. Math. {\bf4}(2) (2003), Article 42.

\bibitem{002}
S. S. Dragomir, {\it Gr\"uss' type inequalities for functions of selfadjoint operators in Hilbert spaces}, Ital. J. Pure Appl. Math. {\bf 28} (2011), 205-222.

\bibitem{003}
S. S. Dragomir, {\it \v Ceby\v sev's type inequalities for functions of selfadjoint operators in Hilbert spaces}, Linear Multilinear Algebra. {\bf58}(7) (2010), 805-814.
  

\bibitem{fujii}
J. I. Fujii and E. Kamei, {\it Relative operator entropy in noncommutative information theory}, Math. Japon. {\bf34} (1989), 341--348.

  \bibitem{5}
  S. Furuichi, H. R. Moradi, {\it On further refinements for Young inequalities}, Open Math. {\bf16} (2018), 1478--1482.
  
 \bibitem{1}
  S. Furuichi and H. R. Moradi, {\it Some refinements of classical inequalities}, Rocky Mountain J. Math. {\bf48}(7)  (2018), 2289--2309.


    
    \bibitem{3}
    S. Furuichi, H. R. Moradi and M. Sababheh, {\it New sharp inequalities for operator means}, Linear Multilinear
    Algebra. {\bf67}(8) (2019), 1567--1578.
  
  \bibitem{FYK2005} 
  S. Furuichi, K. Yanagi and K. Kuriyama, {\it A note on operator inequalities of Tsallis relative operator entropy}, Linear Algebra Appl. {\bf407} (2005), 19--31.
  
\bibitem{gruss}
G. Gr\"uss, {\it Uber das maximum des absoluten betrages von $\frac{1}{b-a}\int_{a}^{b}{f\left( x \right)g\left( x \right)dx}-\frac{1}{{{\left( b-a \right)}^{2}}}\int_{a}^{b}{f\left( x \right)dx}\int_{a}^{b}{g\left( x \right)dx}$}, Math. Z. {\bf39} (1935), 215--226.

\bibitem{halmos}
P . R. Halmos, {\it A Hilbert Space Problem Book} , 2nd ed., Springer, New York, 1982.



\bibitem{4}
  I. H. G\"um\"u\c s, H. R. Moradi and M. Sababheh, {\it More accurate operator means inequalities}, J. Math. Anal.
 Appl. {\bf465} (2018), 267--280.
 

 
\bibitem{kittanehnumerical}
F. Kittaneh, {\it Numerical radius inequalities for Hilbert space operators}, Studia Math. {\bf {168}}(1) (2005), 73--80.


 

  
\bibitem{li}
Xin Li, R. N. Mohapatra and R. S. Rodriguez, {Gr\"{u}ss-type inequalities}, J. Math. Anal. Appl. {\bf{267}} (2002), 434--443.
  




\bibitem{mitri2}
 D. S. Mitrinovi\'{c}, J. E. Pe\v{c}ari\'{c} and A. M. Fink, {\it Gr\"{u}ss Inequality. In: Classical and New Inequalities in Analysis}, Mathematics and its Applications (East European Series), {\bf{ 61}} (1993). Springer, Dordrecht. 


  \bibitem{2}
  H. R. Moradi, S. Furuichi, F. C. Mitroi and R. Naseri, {\it An extension of Jensen's operator   inequality and its application to Young inequality}, Rev. R. Acad. Cienc. Exactas F\'\i s. Nat.   Ser. A Mat. {\bf113} (2) (2019), 605--614.




\bibitem{zou}
 L. Zou and Y. Jiang, {\it Improved arithmetic-geometric mean inequality and its application}, J. Math. Inequal. {\bf 9}(1) (2015), 107--111.
 

  

  

\end{thebibliography}
\end{document}